\documentclass[a4paper,10pt]{amsart}
\usepackage[utf8]{inputenc}

\usepackage{amsmath}
\usepackage{amsthm}
\usepackage{amssymb}

\usepackage{mathrsfs}
\usepackage[all]{xy}

\newcommand{\QQ}{\mathbb{Q}}

\newcommand{\FF}{\mathcal{F}}

\newcommand{\MC}{\mathsf{MC}}
\newcommand{\GG}{\mathcal{G}}

\newcommand{\QC}{\mathcal{C}}

\newcommand{\rQC}{\overline{\mathcal{C}}}
\newcommand{\LL}{\mathbb{L}}

\newcommand{\map}{map}
\newcommand{\aut}{aut}
\newcommand{\Hom}{\operatorname{Hom}}
\newcommand{\Der}{\operatorname{Der}}
\newcommand{\Coder}{\operatorname{Coder}}
\newcommand{\ad}{\operatorname{ad}}
\newcommand{\tensor}{\otimes}

\newcommand{\Fib}[2]{Fib(#2,#1)}

\newcommand{\gl}{\mathfrak{g}}

\newtheorem{theorem}{Theorem}[section]

\newtheorem{proposition}[theorem]{Proposition}
\newtheorem{corollary}[theorem]{Corollary}
\newtheorem{lemma}[theorem]{Lemma}

\theoremstyle{definition}
\newtheorem{definition}[theorem]{Definition}
\newtheorem{remark}[theorem]{Remark}

\title{Rational models for automorphisms of fiber bundles}
\author{Alexander Berglund}
\address{Department of Mathematics\\
Stockholm University\\
SE-106 91 Stockholm\\
Sweden}
\email{alexb@math.su.se}

\begin{document}

\begin{abstract}
Given a fiber bundle, we construct a differential graded Lie algebra model for the classifying space of the monoid of homotopy equivalences of the base covered by a fiberwise isomorphism of the total space.
\end{abstract}

\maketitle

\section{Introduction}
Consider a fiber bundle $p\colon E\to X$ with structure group $G$ over a simply connected CW-complex $X$ and let $\aut_\circ(p)$ denote the topological monoid consisting of commutative diagrams
$$
\xymatrix{E\ar[d]^-p \ar[r]^-\varphi & E \ar[d]^-p \\ X \ar[r]^-f & X,}
$$
such that $f$ is homotopic to the identity map of $X$ and $\varphi$ is a fiberwise isomorphism.
The goal of this paper is to construct a differential graded Lie algebra model for the classifying space $B\aut_\circ(p)$ in the sense of Quillen's rational homotopy theory. We assume that $BG$ is a nilpotent space, i.e., that the group $\pi_0(G)$ is nilpotent and acts nilpotently on $\pi_k(G)$ for all $k\geq 1$.

\begin{theorem}
Let $L$ be the minimal Quillen model for $X$ and let $\Pi$ be a dg Lie algebra model for $BG$. Furthermore, let $\tau\colon \QC L \to \Pi$ be a twisting function that models the classifying map of the bundle $\nu\colon X\to BG$.
Then the classifying space $B\aut_\circ(p)$ is rationally homotopy equivalent to the geometric realization of the dg Lie algebra
\begin{equation} \label{eq:tsdp}
\Hom^\tau(\QC L,\Pi)\langle 0\rangle \rtimes_{\tau_*} \big( \Der L \ltimes_{\ad} sL\big)\langle 1\rangle.
\end{equation}
Here, $\QC L$ is the Chevalley-Eilenberg complex of $L$, we use $\langle n \rangle$ to indicate the $n$-connected cover, and the decorations $\tau$ and $\tau_*$ indicate that we take a twisted semi-direct product (see \S\ref{sec:tsdp} and \S\ref{sec:mc}).
\end{theorem}

In many cases of interest, there is an explicit formula for $\tau$ that yields a simplification of the model. For example, we have the following result in the case of complex vector bundles.
\begin{theorem}
Let $\xi$ be an $n$-dimensional complex vector bundle over a simply connected finite CW-complex $X$ and let $L$ be the minimal Quillen model for $X$. Then $B\aut_\circ(\xi)$ is rationally homotopy equivalent to the geometric realization of the dg Lie algebra
$$
\big(H^*(X;\QQ)\tensor \pi_*(\Omega BU(n)) \big)\langle 0 \rangle \rtimes_{\rho_*}  \big( \Der L \ltimes_{\ad} sL\big)\langle 1\rangle
$$
where $\rho = \sum_{i=1}^n c_i(\xi) \tensor \gamma_i$, 
and $\gamma_i$ is the generator for $\pi_{2i-1}(\Omega BU(n))\tensor \QQ = \pi_{2i}(BU(n))\tensor \QQ$ dual to the universal Chern class $c_i\in H^{2i}(BU(n);\QQ)$.
\end{theorem}

Similar simplifications are possible whenever $G$ is a compact connected Lie group or, more generally, when $H^*(BG;\QQ)$ is a free graded commutative algebra, see \S\ref{sec:applications}.

\begin{remark}
The fibration sequence of dg Lie algebras
$$\Hom^\tau(\QC L,\Pi)\langle 0\rangle \to   \Hom^\tau(\QC L,\Pi)\langle 0\rangle \rtimes_{\tau_*} \big( \Der L \ltimes_{\ad} sL\big)\langle 1\rangle \to \big( \Der L \ltimes_{\ad} sL\big)\langle 1\rangle$$
associated to the twisted semi-direct product \eqref{eq:tsdp} is a model for the homotopy fibration sequence
$$B\aut_X(p) \to B\aut_\circ(p) \to B\aut_\circ(X),$$
where $\aut_\circ(X)$ is the monoid of self-maps of $X$ homotopic to the identity, and $\aut_X(p)\subseteq \aut_\circ(p)$ is the submonoid of $\aut_\circ(p)$ where $f$ is \emph{equal} to the identity on $X$.
In particular, $\Hom^\tau(\QC L,\Pi)\langle 0\rangle$ is a model for $B\aut_X(p)$. We should remark that rational models for $B\aut_X(p)$ have been studied earlier, see e.g.~\cite{Buijs-Smith}.
\end{remark}

\section{Moduli spaces of $\FF$-fibrations}
We will utilize the general framework for classification of fibrations provided by May \cite{May}. Let $(\FF,F)$ be a category of fibers in the sense of \cite[Definition 4.1]{May} and assume it satisfies the hypotheses of the classification theorem \cite[Theorem 9.2]{May}. Also recall the notions of $\FF$-spaces and $\FF$-maps from \cite{BHMP}. Let $G$ denote the group-like topological monoid $\FF(F,F)$, to be thought of as the structure group for $\FF$-fibrations.

In the special case when $\FF$ is the category of all spaces weakly equivalent to a given CW-complex $X$, with morphisms all weak equivalences between such spaces, the `structure group' $G= \aut(X)$ is the monoid of homotopy automorphisms of $X$, and an $\FF$-fibration is the same thing as a fibration with fiber weakly homotopy equivalent to $X$. We will refer to such fibrations as \emph{$X$-fibrations}.

Returning to the general situation, let $p_\infty\colon E_\infty\to B_\infty$ denote the universal $\FF$-fibration, the existence of which is ensured by May's classification theorem, and define
$$
\Fib{\FF}{X} = B\big(\map(X,B_\infty),\aut(X),*\big),
$$
where the right hand side denotes the geometric bar construction of the group-like monoid $\aut(X)$ acting on the space $\map(X,B_\infty)$ from the right by precomposition. It is a consequence of May's `Classification of $Y$-structures' \cite[\S11]{May} that $\Fib{\FF}{X}$ may be thought of as a moduli space of $\FF$-fibrations with base weakly equivalent to $X$. More precisely, we have the following:
\begin{theorem}
For a CW-complex $A$, there is a bijective correspondence between homotopy classes of maps
$$A\to \Fib{\FF}{X}$$
and equivalence classes of $X$-fibrations $p\colon E\to A$ with a $B_\infty$-structure $\theta\colon E\to B_\infty$.
\end{theorem}

\begin{proof}
This follows readily from \cite[Theorem 11.1]{May}.
\end{proof}

In particular, since an $X$-fibration over a point is just a space weakly equivalent to $X$, we see that the set of path components,
$$\pi_0\Fib{\FF}{X},$$
is in bijective correspondence with the set of equivalence classes of $\FF$-fibrations with base weakly homotopy equivalent to $X$.

\begin{definition}
Given an $\FF$-fibration $p\colon E\to B$, let $\aut^\FF(p)$ denote the space of $\FF$-self equivalences of $p$, i.e., the topological monoid 
consisting of commutative diagrams
$$
\xymatrix{E\ar[d]^-p \ar[r]^-\varphi & E \ar[d]^-p \\ X \ar[r]^-f & X,}
$$
such that $f$ is a weak homotopy equivalence and $\varphi$ is a fiberwise $\FF$-map, topologized as a subset of $\map(B,B)\times \map(E,E)$. Let $\aut_\circ^\FF(p) \subseteq \aut^\FF(p)$ denote the submonoid consisting of those pairs $(f,\varphi)$ such that $f$ is homotopic to the identity map on $X$. If $D\subseteq C\subseteq X$ are subsets, then let $\aut_{C}^\FF(p)$ denote the submonoid consisting of pairs as above such that $f$ restricts to the identity map on $C$, and write $\aut_{C}^{D,\FF}(p)$, or simply $\aut_C^D(p)$, for the submonoid of $\aut_C^\FF(p)$ where $\varphi$ restricts to the identity isomorphism on the fibers over points in $D$. Finally, let $\aut_{C,\circ}^D(p)$ denote $\aut_{C}^D(p) \cap \aut_{\circ}^\FF(p)$.
\end{definition}

By using standard properties of the geometric bar construction, we can obtain information about the homotopy types of the components of $\Fib{\FF}{X}$.
\begin{theorem}
\begin{enumerate}
\item There is a bijection
$$\pi_0\Fib{\FF}{X} \cong [X,B_\infty]/\pi_0\aut(X).$$

\item There is a weak equivalence of spaces over $B\aut(X)$,
$$\Fib{\FF}{X} \sim \coprod_{[p]} B\aut^\FF(p),$$
where the union is over all equivalence classes of $\FF$-fibrations $p\colon E\to B$, with $B$ weakly equivalent to $X$.
\end{enumerate}
\end{theorem}

\begin{proof}
As follows from \cite[Proposition 7.9]{May}, there is a homotopy fiber sequence
\begin{equation*} \label{eq:fibration}
\aut(X) \to \map(X,B_\infty) \to \Fib{\FF}{X} \to B\aut(X).
\end{equation*}
The first statement follows by looking at the induced long exact sequence of homotopy groups.

The space of $\FF$-maps $\map^\FF(p,p_\infty)$ is weakly contractible for every $\FF$-fibration $p\colon E\to X$ by \cite[Proposition 3.1]{BHMP}. Consider the diagram
$$
\xymatrix{\aut^\FF(p) \ar[d] \ar[r] & \map^\FF(p,p_\infty) \ar[d] \ar[r] & B\big(\map^\FF(p,p_\infty),\aut^\FF(p),*\big) \ar[d] \ar[r] & B\aut^\FF(p) \ar[d] \\
\aut(X) \ar[r] & \map(X,B_\infty) \ar[r] & B\big(\map(X,B_\infty),\aut(X),*\big) \ar[r] & B\aut(X).}
$$
According to \cite[Proposition 7.9]{May} the rows are quasifibration sequences. The leftmost square is homotopy cartesian. It follows that the third vertical map from the left induces a weak equivalence between the connected components containing $\nu$.
Since $\map^\FF(p,p_\infty)$ is weakly contractible, the rightmost map in the top row is a weak homotopy equivalence. The rightmost square yields a zig-zag of weak homotopy equivalences showing $B\aut^\FF(p) \sim B\big(\map(X,B_\infty),\aut(X),*\big)_\nu$ as spaces over $B\aut(X)$, where $\nu$ indicates the component containing (the class of) $\nu$.
\end{proof}

\begin{corollary} \label{cor:bar model}
There are weak homotopy equivalences
$$B\aut^\FF(p) \sim B\big(\map(X,B_\infty)_\nu,\aut(X)_{[\nu]},*\big),$$
$$B\aut_\circ^\FF(p) \sim B\big(\map(X,B_\infty)_\nu,\aut_\circ(X),*\big),$$
where $\aut(X)_{[\nu]}$ denotes the monoid of homotopy equivalences $\varphi\colon X\to X$ such that $\nu \circ\varphi \simeq \nu$ and $\map(X,B_\infty)_\nu$ denotes the component of $\nu$.
\end{corollary}

\begin{proof}
We have just seen that $B\aut^\FF(p)\sim B\big(\map(X,B_\infty),\aut(X),*\big)_\nu$. The latter is easily seen to be weakly equivalent to $B\big(\map(X,B_\infty)_\nu,\aut(X)_{[\nu]},*\big)$.
The second claim is proved similarly.
\end{proof}

\section{Rational models}
This section contains the proof of the main theorem. We begin by examining the effect of $\QQ$-localization on the geometric bar construction.
Then we will construct a dg Lie model for the $\QQ$-localized bar construction, by combining Schlessinger-Stasheff's \cite{SS} and Tanr\'e's \cite{Tanre} theory of fibrations of dg Lie algebras with Quillen's theory of principal dg coalgebra bundles \cite{Quillen}.

\subsection{Rationalization}
\begin{lemma} \label{lemma:rationalization}
Let $X$ be a connected nilpotent finite CW-complex, let $Z$ be a connected nilpotent space, and fix a map $\nu\colon X\to Z$.
Then $B\big(\map(X,Z)_\nu,\aut_\circ(X),*\big)$ is rationally homotopy equivalent to
$$B\big(\map(X_\QQ,Z_\QQ)_{\nu_\QQ},\aut_{\circ}(X_\QQ),*\big).$$
\end{lemma}

\begin{proof}
By using a functorial $\QQ$-localization for nilpotent spaces, e.g., the Bousfield-Kan $\QQ$-completion, we can construct a commutative diagram
$$\xymatrix{X \ar[r]^-\nu \ar[d]^-{r} & Z \ar[d]^-{q} \\ X_\QQ \ar[r]^-{\nu_\QQ} & Z_\QQ,}$$
where the vertical maps are $\QQ$-localizations. We may also assume that $r$ is a cofibration. Define the monoid $\aut_\circ(r)$ as the pullback
$$
\xymatrix{\aut_\circ(r) \ar@{->>}[r]^-{\sim} \ar[d]^-{\sim_\QQ} & \aut_\circ(X) \ar[d]_-{\sim_\QQ}^-{r_*} \\ \aut_\circ(X_\QQ) \ar@{->>}[r]_-{r^*}^-{\sim} & \map(X,X_\QQ)_{r}.}
$$
Thus, the monoid $\aut_\circ(r)$ consists of pairs $(f,g)$ where $f$ and $g$ are self-maps homotopic to the identity of $X$ and $X_\QQ$, respectively, such that $r\circ f = g\circ r$. Since $r$ is a cofibration, the map $r^*$ is a fibration. It is also a weak equivalence by standard properties of $\QQ$-localization. The map $r_*$ is a rational homotopy equivalence by \cite[Theorem II.3.11]{HMR}. It follows that the projections from $\aut_\circ(r)$ to $\aut_\circ(X)$ and $\aut_\circ(X_\QQ)$ are a weak equivalence and a rational homotopy equivalence, respectively.

There are right actions of the monoid $\aut_\circ(r)$ on $\map(X,Z)$ and $\map(X_\QQ,Z_\QQ)$ through the projections to $\aut_\circ(X)$ and $\aut_\circ(X_\QQ)$, respectively. We get a zig-zag of rational homotopy equivalences of right $\aut_\circ(r)$-spaces
$$\map(X,Z)_\nu \xrightarrow{q_*} \map(X,Z_\QQ)_{q\nu} \xleftarrow{r^*} \map(X_\QQ,Z_\QQ)_{\nu_{\QQ}}.$$
This accounts for the top horizontal zig-zag in the following diagram, where we write $\bullet$ instead of $B\big(\map(X,Z_\QQ)_{q\nu},\aut_\circ(r),*\big)$ to save space,
$$
\xymatrix{B\big(\map(X,Z)_\nu,\aut_\circ(r),*\big)  \ar[d]^-\sim \ar[r]^-{\sim_\QQ} & {\bullet} & \ar[l]_-{\sim_\QQ} B\big(\map(X_\QQ,Z_\QQ)_{\nu_\QQ},\aut_\circ(r),*\big) \ar[d]^-{\sim_\QQ} \\
B\big(\map(X,Z)_\nu,\aut_\circ(X),*\big) && B\big(\map(X_\QQ,Z_\QQ)_{\nu_\QQ},\aut_\circ(X_\QQ),*\big).}$$
\end{proof}

\subsection{Geometric realization of dg Lie algebras}
Let $\gl$ be a dg Lie algebra over $\QQ$, possibly unbounded as a chain complex. For $n\geq 0$, the \emph{$n$-connected cover} is the dg Lie subalgebra $\gl\langle n\rangle \subseteq \gl$ defined by
$$
\gl\langle n \rangle_i = \left\{ \begin{array}{ll} \gl_i, & i>n, \\ \ker(\gl_n \xrightarrow{d} \gl_{n-1}), & i = n, \\ 0, & i<n. \end{array} \right.
$$
We call $\gl$ \emph{connected} if $\gl = \gl\langle 0\rangle$ and \emph{simply connected} if $\gl = \gl\langle 1\rangle$.

The lower central series of $\gl$ is the descending filtration
$$\gl = \Gamma^1\gl \supseteq \Gamma^2 \gl \supseteq \cdots$$
characterized by $\Gamma^1 \gl = \gl$ and $[\Gamma^k \gl,\gl] = \Gamma^{k+1} \gl$.
We call $\gl$ \emph{nilpotent} if the lower central series terminates \emph{degree-wise}, meaning that for every $n$, there is a $k$ such that $(\Gamma^k\gl)_n = 0$. This definition of nilpotence mirrors the notion of nilpotence for topological spaces. Indeed, a connected dg Lie algebra $\gl$ is nilpotent if and only if the Lie algebra $\gl_0$ is nilpotent and the action of $\gl_0$ on $\gl_n$ is nilpotent for all $n$. And, clearly, every simply connected dg Lie algebra is nilpotent.

If $\gl$ is an ordinary nilpotent Lie algebra, then $\exp(\gl)$ denotes the nilpotent group whose underlying set is $\gl$ and where the group operation is given by the Campbell-Baker-Hausdorff formula, see e.g.~\cite{Quillen}. The following generalizes this to dg Lie algebras. Let $\gl$ be a connected nilpotent dg Lie algebra. If $\Omega$ is a commutative cochain algebra, then the chain complex $\gl\tensor \Omega$ becomes a dg Lie algebra with
$$[x\tensor \alpha,y\tensor \beta] = (-1)^{|\alpha||y|}[x,y]\tensor \alpha\beta$$
for $x,y\in \gl$ and $\alpha,\beta\in \Omega$. If $\Omega^k = 0$ unless $0\leq k\leq n$ for some $n$, then the degree $0$ component of $\gl\tensor \Omega$ decomposes as
$$(\gl \tensor \Omega)_0 = (\gl_0 \tensor \Omega^0) \oplus (\gl_1 \tensor \Omega^1) \oplus \cdots \oplus (\gl_n \tensor \Omega^n).$$
From the fact that $[\gl_i\tensor \Omega^i,\gl_j\tensor \Omega^j] \subseteq \gl_{i+j}\tensor \Omega^{i+j}$ and that $\gl_0$ acts nilpotently on $\gl_k$ for all $k$, one sees that $(\gl \tensor \Omega)_0$ is a nilpotent Lie algebra. Hence, so is the Lie subalgebra of zero-cycles $Z_0(\gl\tensor \Omega)$.

Let $\Omega_\bullet$ be the simplicial commutative differential graded algebra where $\Omega_n$ is the Sullivan-de Rham algebra of polynomial differential forms on the $n$-simplex, see \cite{FHT-RHT}. Since $\Omega_n^k = 0$ unless $0\leq k\leq n$, the above construction may be applied levelwise to the simplicial dg Lie algebra $\gl\tensor \Omega_\bullet$.

\begin{definition}
Let $\gl$ be a connected nilpotent dg Lie algebra. We define $\exp_\bullet(\gl)$ to be the simplicial nilpotent group
$$\exp_\bullet(\gl) = \exp Z_0(\gl \tensor \Omega_\bullet).$$
\end{definition}

Next, we recall the definition of the nerve $\MC_\bullet(\gl)$ of a dg Lie algebra $\gl$. As we will see below, the nerve $\MC_\bullet(\gl)$ is a delooping of the simplicial group $\exp_\bullet(\gl)$.

\begin{definition}
A \emph{Maurer-Cartan element} in $\gl$ is an element $\tau$ of degree $-1$ such that
$$d(\tau) + \frac{1}{2}[\tau,\tau] = 0.$$
The set of Maurer-Cartan elements is denoted $\MC(\gl)$. The \emph{nerve} of $\gl$ is the simplicial set
$$\MC_\bullet(\gl) = \MC(\gl \tensor \Omega_\bullet).$$
Define the \emph{geometric realization} of a dg Lie algebra to be the geometric realization of its nerve,
$$|\gl| = |\MC_\bullet(\gl)|.$$
\end{definition}

\subsection{Geometric realization of dg coalgebras}
Let $\Omega$ be a commutative cochain algebra over $\QQ$. A dg coalgebra over $\Omega$ is a coalgebra in the symmetric monoidal category of $\Omega$-modules, i.e., a dg $\Omega$-module $C$ together with a coproduct and a counit,
$$\Delta\colon C\to C\tensor_\Omega C,\quad \epsilon \colon C\to \Omega,$$
such that the appropriate diagrams commute. We let $dgc(\Omega)$ denote the category of dg coalgebras over $\Omega$.
If $C$ is a dg coalgebra over $\Omega$, we let
$$\GG(C)$$
denote the set of group-like elements, i.e., elements $\xi \in C$ of degree $0$ such that
$$\Delta(\xi) = \xi \tensor \xi,\quad d(\xi) = 0,\quad \epsilon(\xi) = 1.$$
Given a dg coalgebra $C$ over $\QQ$, the free $\Omega$-module $C\tensor \Omega$ is a dg coalgebra over $\Omega$. Clearly,
$$\Omega \mapsto \GG(C\tensor \Omega)$$
defines a functor from commutative cochain algebras to sets.

\begin{definition}
Let $C$ be a dg coalgebra. We defined the \emph{spatial realization} of $C$ to be the simplicial set
$$\langle C \rangle = \GG(C\tensor \Omega_\bullet).$$
\end{definition}

A dg Lie algebra over $\Omega$ is a dg $\Omega$-module $L$ together with a Lie bracket $\ell \colon L\tensor_\Omega L \to L$ satisfying the usual anti-symmetry and Jacobi relations. Quillen's generalization of the Chevalley-Eilenberg construction can be extended to dg Lie algebras over $\Omega$, yielding a functor
$$\QC_\Omega \colon dgl(\Omega) \to dgc(\Omega).$$
The underlying coalgebra $\QC_\Omega(L)$ is the symmetric coalgebra $S_\Omega(sL)$, where
$$S_\Omega(V) = \bigoplus_{k\geq 0} (V^{\tensor_\Omega k})_{\Sigma_n},$$
for an $\Omega$-module $V$. The differential is defined as usual, see e.g., \cite[p.301]{FHT-RHT}.
If $L$ is a dg Lie algebra over $\QQ$, then $L\tensor \Omega$ is a dg Lie algebra over $\Omega$ and there is a natural isomorphism of dg coalgebras over $\Omega$,
$$\QC_\Omega(L\tensor \Omega) \cong \QC(L)\tensor \Omega.$$

\begin{proposition} \label{prop:mc}
Let $L$ be a connected dg Lie algebra and let $\Omega$ be a bounded commutative cochain algebra. There is a natural bijection
$$e\colon \MC(L\tensor \Omega) \cong \GG\big(\QC_\Omega(L\tensor \Omega)\big),$$
$$e(\tau) = \sum_{k\geq 0} \frac{1}{k!} s\tau^{\wedge_\Omega k} \in \QC_\Omega(L\tensor \Omega).$$
\end{proposition}

\begin{proof}
The crucial observation is that this series converges since $\Omega$ is bounded and $L$ is connected. Say $\Omega^k = $ unless $0\leq k \leq n$. Then
$$\big(L\tensor \Omega \big)_{-1} = L_0\tensor \Omega^1 \oplus \cdots \oplus L_{n-1} \tensor \Omega^n ,$$
whence $\tau \in L\tensor \Omega^+$ for every element $\tau$ of degree $-1$. Since $(\Omega^+)^k = 0$ for $k>n$, this implies that
$$s\tau \wedge_\Omega \cdots\wedge_\Omega s\tau  = 0$$
whenever there are more than $n$ factors. Clearly, $\Delta(e(\tau)) =e (\tau)\tensor e(\tau)$ and $\epsilon(e(\tau)) = 0$. As the reader may check, the equation $d(e(\tau)) = 0$ is equivalent to the Maurer-Cartan equation for $\tau$.
\end{proof}

\begin{corollary}
There is a natural isomorphism of simplicial sets,
$$\MC_\bullet(L) \cong \langle \QC(L) \rangle,$$
for every connected dg Lie algebra $L$.
\end{corollary}

\begin{proof}
Indeed, $\MC_\bullet(L) = \MC(L\tensor \Omega_\bullet) \cong \GG\big(\QC_\Omega(L\tensor \Omega_\bullet)\big) \cong \GG\big(\QC(L)\tensor \Omega_\bullet\big)$.
\end{proof}

Recall that for a commutative dg algebra $A$, the spatial realization is defined by
$$\langle A\rangle = \Hom_{dga}(A,\Omega_\bullet),$$
see e.g., \cite{Berglund}.
We use the same notation as for the coalgebra realization, but it should be clear from the context which one is used.

\begin{proposition} \label{prop:dga}
Let $A$ be a commutative cochain algebra of finite type with dual dg coalgebra $A^\vee$. Then there is a natural isomorphism
$$\langle A^\vee \rangle \cong \langle A\rangle.$$
\end{proposition}

\begin{proof}
For a bounded commutative cochain algebra $\Omega$ and a finite type dg algebra $A$, there is a natural isomorphism of chain complexes
$$A^\vee \tensor \Omega \cong \Hom(A,\Omega).$$
Under this isomorphism, group-like elements in the dg coalgebra $A^\vee \tensor \Omega$ correspond to morphisms of dg algebras $A\to \Omega$.
\end{proof}

Note that the spatial realization of dg coalgebras preserves products, $\langle C\tensor D \rangle \cong \langle C \rangle \times \langle D \rangle$. In particular, since the universal enveloping algebra $U\gl$ of a dg Lie algebra $\gl$ is a dg Hopf algebra, i.e., a group object in the category of dg coalgebras, its spatial realization $\langle U\gl \rangle$ is a simplicial group.
We also remark that for every commutative cochain algebra $\Omega$, the forgetful functor $dga(\Omega)\to dgl(\Omega)$ admits a left adjoint $U_\Omega\colon dgl(\Omega)\to dga(\Omega)$.

\begin{proposition} \label{prop:exp}
Let $\gl$ be a simply connected dg Lie algebra. There is a natural isomorphism of simplicial groups
$$\exp_\bullet(\gl)\cong \langle U\gl \rangle.$$
\end{proposition}

\begin{proof}
Let $\Omega$ be a bounded commutative cochain algebra, say $\Omega^k = 0$ unless $0\leq k\leq n$.
Observe that there is a canonical isomorphism $U\gl \tensor \Omega \cong U_{\Omega}(\gl\tensor \Omega)$.
The isomorphism is effected by the exponential map
$$\exp\colon Z_0(\gl\tensor \Omega) \to \GG U_\Omega(\gl\tensor \Omega),$$
$$\exp(x) = \sum_{k\geq 0} \frac{1}{k!} x^k,$$
where the product $x^k$ is taken in  $U_{\Omega}(\gl\tensor \Omega)$. The crucial point is that the sum converges.
Indeed, since $\gl$ is simply connected,
$$\big(\gl\tensor \Omega\big)_0 = \gl_1 \tensor \Omega^1 + \cdots + \gl_n\tensor \Omega^n,$$
so $x\in \gl\tensor \Omega^+$, whence $x^k = 0$ for $k>n$, whenever $x$ is an element of degree $0$.
The fact that $\exp$ respects the group structure is essentially by design of the Campbell-Baker-Hausdorff group structure.
\end{proof}

\subsection{Principal dg coalgebra bundles}
Next, recall Quillen's theory of principal dg coalgebra bundles \cite[Appendix B, \S5]{Quillen}. In particular, recall that $\QC(\gl)$ serves as a classifying space for principal $\gl$-bundles. Quillen's universal principal $\gl$-bundle may be identified with
$$U\gl \to \QC(U\gl;\gl)\to \QC(\gl),$$
where $U\gl$ is the universal enveloping algebra of $\gl$ and $\QC(U\gl;\gl)$ is the Chevalley-Eilenberg complex of $\gl$ with coefficients in the right $\gl$-module $U\gl$.

\begin{theorem} \label{thm:principal bundle}
Let $\gl$ be a simply connected dg Lie algebra of finite type. The realization of the universal principal $\gl$-bundle,
$$\langle U\gl \rangle\to \langle \QC(U\gl;\gl)\rangle \to \langle \QC(\gl) \rangle,$$
is a universal principal $\langle U\gl \rangle$-bundle.
\end{theorem}

\begin{proof}
This is proved in \cite[Chapter 25]{FHT-RHT}. Indeed, when $\gl$ is simply connected and of finite type, the coalgebra realization of $U\gl$ is the same as the algebra realization of the dual dg algebra $U\gl^\vee$.
\end{proof}

\begin{corollary} \label{cor:deloop}
Let $\gl$ be a simply connected dg Lie algebra of finite type. The nerve $\MC_\bullet(\gl)$ is a delooping of the simplicial group $\exp_\bullet(\gl)$.
\end{corollary}

\begin{proof}
We have the isomorphisms $\exp_\bullet(\gl) \cong \langle U \gl \rangle$ and $\MC_\bullet(\gl) \cong \langle \QC(\gl) \rangle$.
\end{proof}

\begin{remark}
Since we work with coalgebras, the finite type hypothesis on $\gl$ can be dropped in Theorem \ref{thm:principal bundle} and Corollary \ref{cor:deloop}. However, we will not repeat the lengthy argument here since $\gl$ will be of finite type in our applications.
\end{remark}

\subsection{Twisted semi-direct products and Borel constructions} \label{sec:tsdp}
We begin by recalling certain aspects of Tanr\'e's classification of fibrations in the category of dg Lie algebras \cite[Chapitre VII]{Tanre}.

\begin{definition} \label{def:outer action}
Let $\gl$ and $L$ be dg Lie algebras. An \emph{outer action} of $\gl$ on $L$ consists of two maps
$$\alpha\colon L\tensor \gl \to L,\quad \xi \colon \gl \to L,$$
satisfying the following conditions for all $x,y\in \gl$ and $a,b\in L$, where we write
$$a\ldotp x = \alpha(a\tensor x),\quad x \ldotp a = - (-1)^{|a||x|} a\ldotp x.$$
Firstly, the map $\alpha$ defines an action of $\gl$ on $L$ by derivations, i.e.,
$$[x,y]\ldotp a = x\ldotp ( y\ldotp a) -(-1)^{|x||y|} y\ldotp (x\ldotp a),$$
$$x\ldotp [a,b] = [x \ldotp a,b] + (-1)^{|x||a|}[a,x\ldotp b].$$
Secondly, the map $\xi$ is a chain map of degree $-1$ and a derivation, i.e.,
$$d\xi(x) = -\xi(dx),$$
$$\xi[x,y] = \xi(x)\ldotp y + (-1)^{|x|}x\ldotp \xi(y).$$
Finally, the action and $\xi$ are connected by the equation
$$d(x\ldotp a) = d(x)\ldotp a + (-1)^{|x|}x\ldotp d(a) + [\xi(x),a].$$
\end{definition}

\begin{definition} \label{def:tsdp}
Given an outer action of $\gl$ on $L$, the \emph{twisted semi-direct product} $L\rtimes_\xi \gl$ is the dg Lie algebra
whose underlying graded Lie algebra is the semi-direct product of $\gl$ acting on $L$,
$$\big[(a,x),(b,y)\big] = \big( [a,b] + x\ldotp b + a\ldotp y,[x,y]\big),$$
and whose differential is twisted by $\xi$ in the sense that
$$\partial^\xi(a,x) = (da + \xi(x),dx).$$
\end{definition}
The twisted semi-direct product is the total space in a short exact sequence (i.e.~ fibration sequence) of dg Lie algebras,
\begin{equation} \label{eq:fibration sequence}
0 \to L \to L\rtimes_\xi \gl \to \gl \to 0.
\end{equation}
The section $\gl \to L\rtimes_\xi \gl$, $x\mapsto (0,x)$, is a morphism of graded Lie algebras, but it commutes with differentials if and only if $\xi = 0$.

Outer actions on $L$ are classified by the dg Lie algebra
$$\Der L\ltimes_{\ad} sL,$$
whose underlying graded Lie algebra is the semi-direct product of $\Der L$ acting on the abelian dg Lie algebra $sL$ from the left by
$$\theta \ldotp sx = (-1)^{|\theta|} s \theta(x),$$
and whose differential is given by
$$\partial\big(\theta,sx\big) = \big(\partial(\theta) + \ad_x,-sd(x)\big),$$
where $\ad_x\in \Der L$ is given by $\ad_x(y) = [x,y]$.

\begin{proposition} \label{prop:outer actions}
Specifying an outer action of $\gl$ on $L$ is tantamount to specifying a morphism of dg Lie algebras
$$\phi \colon \gl \to \Der L \ltimes_{\ad} sL.$$
The correspondence is given by
$$\phi(x) = \big(\theta_x,-s\xi(x)\big),$$
where $\theta_x(a) = x\ldotp a$.
\end{proposition}

\begin{proof}
The proof is a straightforward calculation.
\end{proof}

An outer action of $\gl$ on $L$ defines an action of $\gl$ on $\QC(L)$ by coderivations by the following formula:
\begin{align*}
(sa_1 \wedge \cdots \wedge sa_n)\ldotp x & = sa_1\wedge \cdots \wedge sa_n \wedge s\xi(x) \\
& + \sum_{i=1}^n \pm sa_1 \wedge \cdots \wedge s(a_i\ldotp x) \wedge \cdots \wedge sa_n.
\end{align*}
Equivalently, $\QC(L)$ becomes a $U\gl$-module coalgebra, i.e., a right $U\gl$-module such that the structure map $\QC(L)\tensor U\gl \to \QC(L)$ is a morphism of dg coalgebras.

The action of $\Der L\ltimes_{\ad} sL$ by coderivations on $\QC(L)$, derived from the tautological outer action on $L$, gives rise to a morphism of dg Lie algebras that we will denote
\begin{equation} \label{eq:chi}
\chi \colon \Der L\ltimes_{\ad} sL \to \Coder \QC(L).
\end{equation}

\begin{theorem} \label{thm:borel}
Let $\gl$ be a simply connected dg Lie algebra of finite type with an outer action on a connected dg Lie algebra $L$.
There is a right action of the simplicial group $G = \exp_\bullet(\gl)$ on $\MC_\bullet(L)$
and a weak equivalence of simplicial sets over $\MC_\bullet(\gl)$,
$$\MC_\bullet(L\rtimes_\xi \gl) \sim \MC_\bullet(L) \times_G EG.$$
\end{theorem}

\begin{proof}
The action of $\gl$ on $\QC(L)$ makes $\QC(L)$ into a right $U\gl$-module coalgebra. This yields a right action of $\exp_\bullet(\gl) \cong \langle U\gl \rangle$ on $\MC_\bullet(L) \cong \langle \QC(L)\rangle$. The key observation, which may be checked by hand, is that there is an isomorphism of dg coalgebras
$$\QC(L\rtimes_\xi \gl) \cong \QC(\QC(L);\gl).$$
Secondly, we have the standard isomorphism
$$\QC(\QC(L);\gl) \cong \QC(L) \tensor_{U\gl} \QC(U\gl;\gl).$$
By combining these isomorphisms and taking realizations, we get isomorphisms of simplicial sets
$$\MC_\bullet(L\rtimes_\xi \gl) \cong \langle \QC(L\rtimes_\xi \gl) \rangle \cong \langle \QC(L) \tensor_{U\gl} \QC(U\gl;\gl) \rangle \cong  \langle\QC(L) \rangle \times_{\langle U\gl \rangle} \langle \QC(U\gl;\gl)\rangle  .$$
By Theorem \ref{thm:principal bundle}, the simplicial set $\langle \QC(U\gl;\gl)\rangle$ is a model for $EG$. This finishes the proof.
\end{proof}

Let $L$ be a simply connected cofibrant dg Lie algebra of finite type with geometric realization
$$X = |\MC_\bullet(L)|,$$
and consider the simply connected dg Lie algebra
$$\gl = \big(\Der L \ltimes_{\ad} sL \big)\langle 1 \rangle,$$
with associated topological group
$$G = |\exp_\bullet(\gl)|.$$
There is an evident outer action of $\gl$ on $L$, whence an action of the simplicial group $\exp_\bullet(\gl)$ on the nerve $\MC_\bullet(L)$, cf.~Theorem \ref{thm:borel}, whence an action of $G$ on $X$. Since $\gl$ is simply connected, the simplicial group $\exp_\bullet(\gl)$ is reduced, i.e., has only one vertex. In particular, the topological group $G$ is connected. Therefore, the action yields a map of grouplike monoids
\begin{equation} \label{eq:aut}
G \to \aut_\circ(X).
\end{equation}
This map is a weak homotopy equivalence, as follows from, e.g., Tanr\'e's theory \cite[Chapitre VII]{Tanre}.

\subsection{Twisting functions and mapping spaces} \label{sec:mc}
Let $C$ be a dg coalgebra with coproduct $\Delta\colon C\to C\tensor C$ and let $L$ a dg Lie algebra with Lie bracket $\ell\colon L\tensor L\to L$. Recall that a \emph{twisting function} $\tau\colon C\to L$ is a Maurer-Cartan element in the dg Lie algebra $\Hom(C,L)$, whose differential and Lie bracket are given by
$$\partial(f) = d_L \circ f- (-1)^{|f|} f\circ d_C,$$
$$[f,g] = \ell \circ (f\tensor g) \circ \Delta.$$
If $\tau$ is a twisting function, then $\Hom^\tau(C,L)$ denotes the dg Lie algebra with the same underlying graded Lie algebra but twisted differential
$$\partial^\tau(f) = \partial(f) +[\tau,f].$$
Furthermore, there is an outer action of $\Coder C$ on $\Hom^\tau(C,L)$ given by
$$f\ldotp \theta = f\circ \theta,\quad \xi(\theta) = \tau_*(\theta) = \tau\circ \theta,$$
for $f\in \Hom^\tau(C,L)$ and $\theta\in \Coder C$. We note for future reference that we may make the identification
\begin{equation} \label{eq:twist}
\big(\Hom(C,L)\rtimes \Coder C \big)^\tau = \Hom^\tau(C,L)\rtimes_{\tau_*} \Coder C
\end{equation}
for every twisting function $\tau\colon C\to L$.

\begin{theorem} \label{thm:mapping space}
Let $L$ and $\Pi$ be connected dg Lie algebras and suppose $\Pi$ is nilpotent and of finite type. There is a natural weak homotopy equivalence of simplicial sets
$$\MC\big(\Hom(\QC L,\Pi\tensor \Omega_\bullet)\big) \xrightarrow{\sim} \map\big(\MC_\bullet(L),\MC_\bullet(\Pi)\big).$$
\end{theorem}

\begin{proof}
Let $\Omega$ be a bounded commutative cochain algebra. We define a natural map
\begin{equation} \label{eq:map}
\MC \Hom(\QC(L),\Pi\tensor \Omega) \times \MC(L\tensor \Omega) \to \MC(\Pi \tensor \Omega)
\end{equation}
as follows. First, make the identifications
\begin{align*}
\Hom(\QC(L),\Pi\tensor \Omega) & = \Hom_\Omega(\QC_\Omega(L\tensor \Omega),\Pi \tensor \Omega), \\
\MC(L\tensor \Omega) & = \GG \big(\QC_\Omega(L\tensor \Omega)\big),
\end{align*}
the second of which is justified by Proposition \ref{prop:mc}, and then define
$$\epsilon\colon \MC\Hom_\Omega(\QC_\Omega(L\tensor \Omega),\Pi\tensor \Omega) \times \GG \big(\QC_\Omega(L\tensor \Omega)\big) \to \MC(\Pi\tensor \Omega),$$
simply by evaluation,
$$\epsilon(\tau,\xi) = \tau(\xi).$$
We need to verify that $\tau(\xi)$ satisfies the Maurer-Cartan equation.
Since $\tau$ is a twisting function, it satisfies the equation
$$0 = \partial(\tau) +\frac{1}{2}[\tau,\tau].$$
Evaluating both sides at the group-like element $\xi$ yields
$$0 = d\tau(\xi) +\tau d(\xi) + \frac{1}{2}\ell \circ (\tau\tensor \tau) \circ \Delta(\xi) = d\tau(\xi) + \frac{1}{2} [\tau(\xi),\tau(\xi)],$$
showing that $\tau(\xi)$ satisfies the Maurer-Cartan equation.

The map is clearly natural in $\Omega$ and yields a simplicial map
$$\MC \Hom(\QC(L),\Pi\tensor \Omega_\bullet) \times \MC(L\tensor \Omega_\bullet) \to \MC(\Pi \tensor \Omega_\bullet).$$
The map in the theorem is defined to be the adjoint of this map.

To show it is a weak homotopy equivalence, one argues as in \cite[Theorem 6.6]{Berglund} by induction on a suitable complete filtration of $\Pi$. The proof is entirely analogous so we omit the details.
\end{proof}

\begin{remark}
The dg Lie algebra $\Hom(\QC L,\Pi)$ with the descending filtration
$$F^{r+1} = \Hom(\QC L, \Pi\langle r\rangle), \quad r\geq 0,$$
is a complete dg Lie algebra in the sense of \cite[Definition 5.1]{Berglund}.
By \cite[Theorem 6.3]{Berglund} (see also Definition 5.3 and Remark 6.4 in \emph{loc.cit.}), the Kan complex
\begin{equation*}
\widehat{\MC}_\bullet(\Hom(\QC L,\Pi)) = \varprojlim  \MC_\bullet \Hom(\QC L,\Pi/\Pi\langle r\rangle)
\end{equation*}
is homotopy equivalent to $\map(\MC_\bullet(L),\MC_\bullet(\Pi))$. We would like to remark how this relates to the statement in Theorem \ref{thm:mapping space}.

Since $\Pi/\Pi\langle r\rangle$ is finite dimensional for all $r$, we have
$$\Hom(\QC L,\Pi/\Pi\langle r\rangle)\tensor \Omega_\bullet \cong \Hom(\QC L,\Pi/\Pi\langle r\rangle \tensor \Omega_\bullet)$$
Upon taking the inverse limit, we get an isomorphism of simplicial sets
$$\widehat{\MC}_\bullet(\Hom(\QC L,\Pi)) \cong \MC_\bullet(\Hom(\QC L,\Pi\tensor \Omega_\bullet)).$$
Thus, Theorem \ref{thm:mapping space} and Theorem 6.3 in \cite{Berglund} say the same thing.
The advantage of Theorem \ref{thm:mapping space} is that the explicit formula for the map gives us control over equivariance properties, as we will see next.
\end{remark}

Let $L$ be a simply connected cofibrant dg Lie algebra of finite type. Precomposition defines a right action of the dg Lie algebra $\Coder \QC(L)$ on the complete dg Lie algebra $\Hom(\QC(L),\Pi)$. By composing with \eqref{eq:chi}, and restricting to the simply connected cover, we get an action of the dg Lie algebra
$$\gl = \big( \Der L \ltimes_{\ad} sL \big)\langle 1 \rangle$$
on $\Hom(\QC(L),\Pi)$. By Theorem \ref{thm:borel}, this induces an action of the simplicial group $\exp_\bullet(\gl)$ on the simplicial set $\MC\Hom(\QC(L),\Pi\tensor\Omega_\bullet) \cong \widehat{\MC}_\bullet(\Hom(\QC(L),\Pi))$.
On the other hand, $\exp_\bullet(\gl)$ acts on $\MC_\bullet(L)$ and hence also on $\map(\MC_\bullet(L),\MC_\bullet(\Pi))$.
The following is an important addendum to Theorem \ref{thm:mapping space}.
\begin{proposition} \label{prop:equivariance}
The weak equivalence of Theorem \ref{thm:mapping space},
$$\MC\Hom(\QC L,\Pi\tensor\Omega_\bullet) \xrightarrow{\sim} \map(\MC_\bullet(L),\MC_\bullet(\Pi)),$$
is equivariant with respect to the action of the simplicial group $\exp_\bullet(\gl)$.
\end{proposition}

\begin{proof}
The proof boils down to the easily checked fact that the map $\epsilon$ in the proof of Theorem \ref{thm:mapping space} satisfies
$$\epsilon(\theta\ldotp f,\xi) = \epsilon(f,\xi\ldotp \theta),$$
for $\theta\in\GG U_{\Omega_\bullet}(\gl\tensor \Omega_\bullet)$, $f\in \MC\Hom_\Omega(\QC_\Omega(L\tensor \Omega),\Pi\tensor \Omega)$ and $\xi\in\GG \big(\QC_\Omega(L\tensor \Omega)\big)$.
\end{proof}

\begin{proposition} \label{prop:bar construction}
Let $X_\QQ$ and $Z_\QQ$ be $\QQ$-local connected nilpotent spaces of finite $\QQ$-type. Let $L$ be a finite type cofibrant dg Lie algebra model for $X_\QQ$ and let $\Pi$ be any dg Lie model for $Z_\QQ$. The geometric bar construction,
$$B\big(\map(X_\QQ,Z_\QQ),\aut_\circ(X_\QQ),*\big),$$
is weakly homotopy equivalent to the geometric realization of the dg Lie algebra
$$\Hom(\QC L,\Pi) \rtimes \big( \Der L\ltimes_{\ad} sL\big)\langle 1 \rangle.$$
\end{proposition}

\begin{proof}
We may as well assume $X_\QQ = \MC_\bullet(L)$ and $Z_\QQ = \MC_\bullet(\Pi)$.
By Theorem \ref{thm:borel}, there is a weak homotopy equivalence
$$\widehat{\MC}_\bullet\big(\Hom(\QC L,\Pi) \rtimes \gl \big) \sim B\big(\widehat{\MC}_\bullet\big(\Hom(\QC L,\Pi)\big),\exp_\bullet(\gl),*\big).$$
The weak equivalence $\exp_\bullet(\gl) \to \aut_\circ(X)$ of group-like simplicial monoids and the weak equivalence of $\exp_\bullet(\gl)$-spaces of Proposition \ref{prop:equivariance} combine to give a weak homotopy equivalence
$$B\big(\widehat{\MC}_\bullet\big(\Hom(\QC L,\Pi)\big),\exp_\bullet(\gl),*\big) \xrightarrow{\sim} B\big(\map(X_\QQ,Z_\QQ),\aut_\circ(X_\QQ),*\big).$$
\end{proof}

\subsection{Proof of the main result}

\begin{theorem} \label{thm:main}
Suppose that $\FF$ is a category of fibers such that the classifying space $B_\infty$ is connected and nilpotent.
Let $p\colon E\to X$ be an $\FF$-fibration over a simply connected finite CW-complex $X$.

Let $L$ be a simply connected cofibrant dg Lie algebra model for $X$ and let $\Pi$ be a connected nilpotent dg Lie algebra model for $B_\infty$. Let $\tau\colon \QC L \to \Pi$ be a twisting function that models the map $\nu \colon X\to B_\infty$ that classifies $p$.

Then the classifying space $B\aut_{\circ}^\FF(p)$ is rationally homotopy equivalent to the geometric realization of the dg Lie algebra
$$\Hom^\tau(\QC L,\Pi)\langle 0\rangle \rtimes_{\tau_*} \big( \Der L\ltimes_{\ad} sL \big)\langle 1\rangle.$$
\end{theorem}

\begin{proof}
For notational convenience, let $Z = B_\infty$. As before, let
$$\gl = \big( \Der L \ltimes_{\ad} sL \big)\langle 1 \rangle.$$
That the dg Lie algebras $L$ and $\Pi$ are models for $X$ and $Z$ means that we may use their geometric realizations as models for the $\QQ$-localizations of $X$ and $Z$;
$$X_\QQ = |\MC_\bullet(L)|,\quad Z_\QQ =|\MC_\bullet(\Pi)|.$$

By Corollary \ref{cor:bar model} and Lemma \ref{lemma:rationalization} we have
$$B\aut_\circ^\FF(p) \sim B\big(\map(X,Z)_\nu,\aut_\circ(X),*\big) \sim_\QQ B\big(\map(X_\QQ,Z_\QQ)_{\nu_\QQ},\aut_\circ(X_\QQ),*\big).$$
The latter space is weakly homotopy equivalent to the component
$$B\big(\map(X_\QQ,Z_\QQ),\aut_\circ(X_\QQ),*\big)_{\nu_\QQ}.$$
By Proposition \ref{prop:bar construction}
$$B\big(\map(X_\QQ,Z_\QQ),\aut_\circ(X_\QQ),*\big) \sim \widehat{\MC}\big(\Hom(\QC L,\Pi)\rtimes \gl\big).$$
Let $\tau\colon \QC L \to \Pi$ be a twisting function that corresponds to $\nu_\QQ$. It follows from \cite[Corollary 1.3]{Berglund} that the component
$$\widehat{\MC}_\bullet\big(\Hom(\QC L,\Pi)\rtimes \gl\big)_\tau \sim \widehat{\MC}_\bullet\big((\Hom(\QC L,\Pi)\rtimes \gl)^\tau\langle 0 \rangle \big).$$
Finally, as in \eqref{eq:twist} one checks that there is an isomorphism of dg Lie algebras
$$(\Hom(\QC L,\Pi)\rtimes \gl)^\tau\langle 0 \rangle \cong \Hom^\tau(\QC L,\Pi)\langle 0 \rangle \rtimes_{\tau_*} \gl.$$
This finishes the proof.
\end{proof}

\begin{remark}
It is straightforward to derive the following variants of the main result:
If $A\subseteq X$ is a subspace such that the inclusion of $A$ into $X$ is a cofibration, then we may consider the submonoid $\aut_{\circ,A}^\FF(p) \subseteq \aut_{\circ}^\FF(p)$ where the homotopy automorphism of the base restricts to the identity map on $A$. If
$$\gl_A \to \big( \Der L \ltimes_{\ad} sL \big)\langle 1 \rangle$$
is a dg Lie algebra morphism that models the map $B\aut_{A,\circ}(X) \to B\aut_\circ(X)$, then
$$\Hom^\tau(\QC L,\Pi)\langle 0\rangle  \rtimes_{\tau_*} \gl_A$$
is a dg Lie algebra model for the space $B\aut_{A,\circ}^\FF(p)$.
Similarly, if we pick a base-point $*\in A\subseteq X$, and let $\aut_{A,\circ}^*(p)$ denote the submonoid of $\aut_{A,\circ}^\FF(p)$ where the automorphism of the total space restricts to the identity over the base-point, then one gets a model for $B\aut_{A,\circ}^*(p)$ by replacing $\QC(L)$ with the \emph{reduced} Chevalley-Eilenberg chains.
\end{remark}

\section{Examples and applications} \label{sec:applications}
Many classifying spaces of interest in geometry have simple rational homotopy types:

\begin{itemize}
\item If $G$ is a compact connected Lie group, then $H^*(BG;\QQ)$ is a polynomial algebra on finitely many generators of even degree (see, e.g., \cite[Theorem 1.81]{FOT}).

\item The stable classifying spaces $BO$, $BTop$, $BPL$ are infinite loop spaces and have rational cohomology rings of the form $\QQ[p_1,p_2,\ldots ]$, where $p_i$ is a generator of degree $4i$ (see, e.g., \cite{MM}).

\item Halperin's conjecture, which has been verified in many cases, asserts that $H^*(B\aut_\circ(X);\QQ)$ is a polynomial algebra whenever $X$ is an elliptic space with non-zero Euler characteristic.
\end{itemize}

In this section, we will provide a simplification of the model arrived at in the previous section in the case when $H^*(B_\infty;\QQ)$ is a polynomial algebra with finitely many generators in each degree.

Call the generators $p_1,p_2,\ldots$ and let $d_i = |p_i|$.
Let $X$ be a simply connected finite CW complex together with an $\FF$-bundle classified by a map
$$\nu \colon X\to B_\infty.$$
The characteristic classes of the bundle are defined by pulling back the universal classes $p_i$ along $\nu$;
$$p_i(\nu) = \nu^*(p_i) \in H^{d_i}(X;\QQ).$$
A dg Lie algebra model for $B_\infty$ is provided by the abelian dg Lie algebra with zero differential
$$\Pi = \pi_*(\Omega B_\infty)\tensor \QQ.$$
This graded vector space is spanned by classes $\pi_i\in \pi_{d_i-1}(\Omega B_\infty)\tensor \QQ = \pi_{d_i}(\Omega B_\infty)\tensor \QQ$ that are dual to $p_i$ under the Hurewicz pairing between cohomology and homotopy.

Let $L$ denote the minimal Quillen model of $X$. Recall that it has the form
$$L  = \big(\LL(V),\delta\big),$$
where $V = s^{-1}\widetilde{H}_*(X;\QQ)$ and the differential $\delta$ is decomposable in the sense that $\delta(L) \subseteq [L,L]$. Thus, the suspension of the space of indecomposables, $sL/[L,L]$, may be identified with $\widetilde{H}_*(X;\QQ)$.

We will work with based maps, so in this section we let $\rQC(L)$ denote the \emph{reduced} Chevalley-Eilenberg chains on a dg Lie algebra $L$. It is defined as the cokernel of the coaugmentation $\eta\colon \QQ\to \QC(L)$.

The restriction of \eqref{eq:chi} to $\Der L$ gives a morphism of dg Lie algebras
$$\chi\colon \Der L \to \Coder \rQC(L).$$
In particular, $\rQC(L)$ is a module over $\Der L$.
Moreover, the universal twisting function $\tau_L\colon \rQC(L) \to L$ is a morphism of $\Der L$-modules.
Indeed, for $\theta\in \Der L$, the coderivation $\chi(\theta) \in \Coder \rQC(L)$ is characterized by commutativity of the diagram
$$
\xymatrix{\rQC(L) \ar[d]^-{\tau_L} \ar[r]^-{\chi(\theta)} & \rQC(L) \ar[d]^-{\tau_L} \\ L \ar[r]^-\theta & L,}
$$
which means that $\tau_L$ is a morphism of $\Der L$-modules.

Consider the composite of the universal twisting function and the abelianization,
$$\rQC(L) \xrightarrow{\tau_L} L \xrightarrow{a} L/[L,L].$$
Since $a$ is a morphism of dg Lie algebras, this composite is a twisting function. But a twisting function with abelian target is the same thing as a chain map of degree $-1$. Thus, for every dg Lie algebra $L$, there is a canonical chain map
$$
g_L\colon \rQC(L) \to sL/[L,L].
$$
Moreover, $g_L$ is a morphism of $\Der L$-modules. Indeed, we have just seen that $\tau_L$ is a morphism of $\Der L$-modules, and $a\colon L\to L/[L,L]$ is obviously a morphism of $\Der L$-modules.

\begin{proposition} \label{prop:Q}
If $L$ is a cofibrant dg Lie algebra, then the canonical map
$$
g_L\colon \rQC(L) \to sL/[L,L]
$$
is a quasi-isomorphism.
\end{proposition}

\begin{proof}
See, e.g., \cite[Proposition 22.8]{FHT-RHT}
\end{proof}

Next, let $L$ be a cofibrant minimal dg Lie algebra model for $X$ and let $\Pi$ denote the abelian graded Lie algebra $\pi_*(\Omega B_\infty)\tensor \QQ$. Consider the degree $-1$ map of graded vector spaces
\begin{align*}
\rho\colon \widetilde{H}_*(X;\QQ) & \to \pi_*(\Omega B_\infty) \tensor \QQ, \\
e & \mapsto \sum_i \langle p_i(\nu),e\rangle \pi_i,
\end{align*}
where $\langle -,- \rangle$ denotes the standard pairing between homology and cohomology (and $\langle p,e \rangle = 0$ unless $e$ and $p$ have the same degree).

\begin{proposition} \label{prop:tau model}
The composite map
$$\tau\colon \rQC(L) \xrightarrow{g_L} sL/[L,L] = \widetilde{H}_*(X;\QQ) \xrightarrow{\rho} \pi_*(\Omega B_\infty) \tensor \QQ$$
is a twisting function that models the map $\nu \colon X\to B_\infty$.
\end{proposition}

\begin{proof}
The map
$$\psi\colon L\xrightarrow{a} L/[L,L] = s^{-1} \widetilde{H}_*(X;\QQ) \xrightarrow{\rho} \pi_*(\Omega B_\infty) \tensor \QQ$$
is a dg Lie model for $\nu \colon X\to B_\infty$. Hence, composing with the universal twisting function we get a twisting function that models $\nu$:
$$\rQC(L) \xrightarrow{\tau_L} L \xrightarrow{\psi} \pi_*(\Omega B_\infty) \tensor \QQ.$$
This composite map is the same as the map in the statement of the proposition.
\end{proof}

\begin{theorem}
The classifying space $B\aut_{*,\circ}^*(\nu)$ is rationally homotopy equivalent to the geometric realization of the dg Lie algebra
$$
\Hom\big(\widetilde{H}_*(X;\QQ),\pi_*(\Omega B_\infty)\tensor \QQ\big)\langle 0\rangle \rtimes_{\rho_*} \Der L\langle 1 \rangle.
$$
\end{theorem}

\begin{proof}
By Proposition \ref{prop:tau model}, the map $\tau = \rho\circ g_L$ is a twisting function that models the map $\nu \colon X\to B_\infty$.
By Theorem \ref{thm:main} the dg Lie algebra
$$\Hom^\tau\big(\rQC L,\Pi\big) \langle 0 \rangle \rtimes_{\tau_*} \Der L \langle 1 \rangle$$
is a model for $B\aut_{*,\circ}^*(\nu)$. Since $\Pi$ is abelian, twisting has no effect, i.e., we have $\Hom^\tau\big(\rQC L,\Pi\big) = \Hom\big(\rQC L,\Pi\big)$.
Since $g_L^*(\rho) = \tau$ by definition of $\tau$, we see that the quasi-isomorphism of Proposition \ref{prop:Q} induces a quasi-isomorphism of dg Lie algebras
$$
g^* \rtimes 1 \colon \Hom(sL/[L,L],\Pi)\langle 0 \rangle \rtimes_{\rho_*} \Der L \langle 1 \rangle \to \Hom(\rQC L,\Pi) \langle 0 \rangle \rtimes_{\tau_*} \Der L\langle 1\rangle.
$$
Finally note that $sL/[L,L] = \widetilde{H}_*(X;\QQ)$.
\end{proof}

\begin{remark}
When $\widetilde{H}_*(X;\QQ)$ is finite dimensional, we can rewrite the result in terms of cohomology since
$$\Hom(\widetilde{H}_*(X;\QQ),\pi_*(\Omega B_\infty)) \cong \widetilde{H}^*(X;\QQ) \tensor \pi_*(\Omega B_\infty).$$
The twisting function takes the form
$$\rho = \sum_i p_i(\nu)\tensor \pi_i \in \widetilde{H}^*(X;\QQ) \tensor \pi_*(\Omega B_\infty)$$
in this case.
\end{remark}

\begin{remark}
Again there are easily proved variants of this result. If $A\subseteq X$ is a subspace containing the base-point such that the inclusion of $A$ into $X$ is a cofibration, then we may consider the submonoid $\aut_{A,\circ}^*(\nu) \subseteq \aut_{*,\circ}^*(\nu)$ where the homotopy automorphism of the base restricts to the identity map on $A$. If $\gl \to \Der L\langle 1\rangle$ is a morphism of dg Lie algebras that models the map $B\aut_A(X) \to B\aut_*(X)$, then
$$\Hom\big(\widetilde{H}_*(X;\QQ),\pi_*(\Omega B_\infty)\tensor \QQ\big)\langle 0\rangle  \rtimes_{\rho_*} \gl$$
is a dg Lie algebra model for the space $B\aut_{A,\circ}^*(\nu)$.
In \cite{BM} we apply this result in the case $(X,A) = (M,\partial M)$ to construct a dg Lie algebra model for the block diffeomorphism group of a simply connected smooth $n$-manifold $M$ with boundary $\partial M = S^{n-1}$ ($n\geq 5$).
\end{remark}

\subsection*{Acknowledgments}
The author thanks Wojciech Chach\'olski for discussions about classification of fibrations and Ib Madsen for numerous discussions. This work was supported by the Swedish Research Council through grant no.~2015-03991.

\end{document}